%% file: WkWave_Final_Submit_II.tex
\documentclass[11pt,reqno]{article}

\usepackage[margin=1in]{geometry}
\usepackage{amsmath,amsthm,amssymb}

\usepackage{esint}

\usepackage[colorlinks=true, pdfstartview=FitV, linkcolor=blue,citecolor=blue, urlcolor=blue]{hyperref}

\usepackage[abbrev,lite,nobysame]{amsrefs}
\usepackage{times}
\usepackage[usenames,dvipsnames]{color}
\usepackage{mathtools,enumitem}

\usepackage[compact]{titlesec}

\usepackage[title]{appendix}

\usepackage{comment}

\input{macros.tex}

\numberwithin{equation}{section}
    
\begin{document}

\title{A note on cascade flux laws for the stochastically-driven nonlinear Schr\"odinger equation} 
\author{Jacob Bedrossian\thanks{\footnotesize Department of Mathematics. University of California, Los Angeles, CA 90095. \href{mailto:jacob@math.ucla.edu}{\texttt{jacob@math.ucla.edu}} This work was supported by NSF Award DMS-2108633.} }

\maketitle

\begin{abstract}
In this note we point out some simple sufficient (plausible) conditions  for `turbulence' cascades in suitable limits of damped, stochastically-driven nonlinear Schr\"odinger equation in a $d$-dimensional periodic box.
Simple characterizations of dissipation anomalies for the wave action and kinetic energy in rough analogy with those that arise for fully developed turbulence in the 2D Navier-Stokes equations are given and sufficient conditions are given which differentiate between a `weak' turbulence regime and a `strong' turbulence regime.
The proofs are relatively straightforward once the statements are identified, but we hope that it might be useful for thinking about mathematically precise formulations of the statistically-stationary wave turbulence problem. 
\end{abstract}

\setcounter{tocdepth}{1}
{\small\tableofcontents}

\section{Introduction}
This note is regarding (weak) wave turbulence in dispersive PDE in the statistically stationary regime subjected to white-in-time stochastic forcing and dissipation.
We are interested in precising some simple sufficient conditions to observe the flux balance laws of both inverse and direct cascades that are observed in wave turbulence; see e.g. \cite{nazarenko2011wave,zakharov2012kolmogorov} and more discussions below.
In this section, we consider a torus of side-length $2\pi \lambda$,  $\mathbb T^d_\lambda$ with $d=2,3$ and the cubic nonlinear Schr\"odinger equation (NLS) rescaled in the following manner for two small parameters $\sigma,\nu > 0$ 
\begin{align}
-i d u = \left(\Delta u + \sigma \abs{u}^2 u + i \nu D(\grad) u\right) dt + i \sigma \sum_j g^\lambda_j dW_t^{(j)}, \label{def:NLS} 
\end{align}
%\textcolor{red}{
The choice of splitting the $\sigma$ between the forcing and the nonlinearity is so that the formal limiting dynamics in the $\nu,\sigma \to 0$ limit is the linear Schr\"odinger equation; see Section \ref{sec:Phys} for more discussion on this point.
Here the $\set{W_t^{(j,\lambda)}}$ are independent Wiener processes on a common canonical filtered probability space $(\Omega,\mathcal{F},\mathcal{F}_t,\PP)$.
We have allowed force to depend on the side-length of the box, i.e. for each $\lambda > 0$, there is a different set of $\set{g^\lambda_j}$ and Brownian motions $W_t^{(j,\lambda)}$. 
The exact assumptions on the forcing is given below in Assumption \ref{Ass:g}; for now we just remark that the functions $g_j^\lambda$ are average zero, localized in frequency at scales $\approx 1$, locally uniformly bounded, i.e. $\sup_{\lambda,j} \norm{g^\lambda_j}_{W^{1,\infty}} < \infty$, and normalized such that the volume-averaged wave-action (mass) input per unit time is $\approx \sigma^2$, namely the noise is fixed so that the following parameter is approximately independent of $\lambda$:
\begin{align*}
\eps_{wa} = \eps_{wa}(\lambda) := \frac{1}{2}\sum_j \fint_{\mathbb T_\lambda^d} \abs{g^\lambda_j}^2 \dx, 
\end{align*}
where the crossed integral sign denotes average, i.e. $\fint_{\mathbb T_\lambda^d} = \frac{1}{(2\pi \lambda)^d} \int_{\mathbb T_\lambda^d}$.
We also define a corresponding $\eps_{ke}$, although its interpretation as related to `volume-averaged kinetic energy input per unit time' will only be precise in a weak turbulence regime
\begin{align*}
\eps_{ke} := \frac{1}{2}\sum_j \fint_{\mathbb T_\lambda^d} \abs{\grad g^\lambda_j}^2 \dx. 
\end{align*}
For simplicity and definiteness we set the dissipation operator $D(k)$ to simply be
\begin{align*}
D = 1-\Delta. 
\end{align*}
With this choice of $D$, \eqref{def:NLS} is almost-surely globally well-posed in both 2D and 3D and the associated Markov process has at least one stationary measure supported on smooth functions \cite{KS04} (see Theorem \ref{lem:NLSGWP} below for precise statement). 
In what follows, we generally use volume-averaged $L^p$ spaces,
\begin{align*}
\norm{f}_{L^p_\lambda} := \left( \fint_{\mathbb T^d_\lambda} \abs{f(x)}^p \dx \right)^{1/p}. 
\end{align*}
The (non-driven, non-damped) conservative initial value problem conserves two basic quantities, the wave action (here volume-averaged):
\begin{align*}
  \mathcal{WA}_\lambda[u]  = \frac{1}{2}\norm{u}_{L^2_\lambda}^2;
\end{align*}
and the Hamiltonian energy: 
\begin{align*}
\mathcal{H}_\lambda[u] = \frac{1}{2}\norm{\grad u}_{L^2_\lambda}^2 + \frac{1}{4} \sigma \norm{u}_{L^4_\lambda}^4. 
\end{align*}
Let $\mu = \mu_{\lambda,\nu,\sigma}$ be an invariant measure of the SPDE \eqref{def:NLS} (see Lemma \ref{lem:NLSGWP}). 
In the sequel we denote
\begin{align*}
\EE \phi := \EE_\mu \phi := \int_{L^2} \phi(u) \mu(du). 
\end{align*}
The white-in-time forcing implies the following balance law (note that because $\mathcal{H}$ is not quadratic, the It\^o correction still depends on $u$).
\begin{lemma}[Balance of wave action and Hamiltonian dissipation \cite{KS04}] 
For all $\lambda,\sigma,\nu>$, all invariant measures $\mu$ of \eqref{def:NLS} satisfy the balance of wave-action dissipation
\begin{align}
\nu \EE \norm{D^{1/2} u}_{L^2_\lambda}^2 = \sigma^2 \eps_{wa}, \label{eq:WABal}
\end{align}
and the balance of Hamiltonian energy dissipation
\begin{align}
\nu \EE \norm{\grad D^{1/2} u}_{L^2_\lambda}^2 + \nu \sigma \EE \norm{u}_{L^4_\lambda}^4 + \nu \sigma 3 \EE \norm{u \grad u}_{L^2_\lambda}^2 = \sigma^2 \eps_{ke} + \frac{\sigma^3}{2}\sum_j  \EE \fint_{\mathbb T_\lambda^d} \abs{u}^2 \abs{g^\lambda_j}^2 \dx. \label{eq:Hbal}
\end{align}
In what follows we use the following terminologies
\begin{align*}
  \nu \EE \norm{D^{1/2} u}_{L^2_\lambda}^2 = \textup{ ``wave-action dissipation''} \\
  \nu \EE \norm{\grad D^{1/2} u}_{L^2_\lambda}^2 + \nu \sigma \EE \norm{u}_{L^4_\lambda}^4 + \nu \sigma 3 \EE \norm{u \grad u}_{L^2_\lambda}^2 =  \textup{ ``Hamiltonian energy dissipation''} \\
  \nu \EE \norm{\grad D^{1/2} u}_{L^2_\lambda}^2 =  \textup{ ``kinetic energy dissipation''}. 
\end{align*}
\end{lemma}
%\textcolor{red}{
The most fundamental concepts in statistical theories of turbulence are the notions of anomalous dissipation and nonlinear cascade, wherein a conserved quantity is injected and then through nonlinear effects, is sent either to lower frequencies (an inverse cascade) or to higher frequencies (a direct cascade) where it is eventually dissipated by the damping on the system, regardless of how small the dissipative parameter $\nu$ is (called `anomalous dissipation').
  The range of scales between the injection scale and the dissipation scale(s) is/are called the \emph{inertial range(s)}\footnote{This terminology refers to intermediate ranges of scales over which little to no external forcing or dissipative effects are directly acting on the solution. Note that this can happen in both large and small scales. The terminology originates in fluid mechanics, as in this range scales, it leaves  only the nonlinear/intertial and pressure effects.}, and it is here that we expect statistical universality, i.e. the statistics should be essentially independent of the exact form of the dissipative effects and external forcing.
The statistical theory of hydrodynamic turbulence in the 3D Navier-Stokes equations was founded in earnest by Kolmogorov in his K41 works \cite{K41a,K41b,K41c} where he gave a (non-mathematically rigorous) derivation of his $4/5$-law and predicted the power spectrum\footnote{The `power spectrum' regards a prediction of the $\nu \to 0$ asymptotics of $\EE \abs{k}^{d-1} \abs{\hat{u}(k)}^2$ (in the case of fluid mechanics, this is the kinetic energy density in a frequency shell).} in the inertial range; see for example \cite{Frisch1995,Eyink94} for more discussion.  
The $4/5$-law describes the on-average constant flux of kinetic energy from the injection scale down to the viscous dissipation scale in a statistically stationary flow. It is considered the only `exact law' of turbulence and an analogous law of constant flux of conserved quantities through the inertial ranges is expected in all turbulent systems (see discussions in \cite{nazarenko2011wave,zakharov2012kolmogorov}). 

\emph{Wave turbulence} in dispersive PDEs is expected to hold in a weakly nonlinear regime of turbulence, i.e. when the solution is sufficiently small for the linear wave dynamics to dominate on large time-scales.  
The classical wave turbulence theory began in \cite{Peierls1929} with the derivation of the wave-kinetic equation (WKE) and was later continued in \cite{H62,H63}; see for example \cite{spohn2008boltzmann,nazarenko2011wave,zakharov2012kolmogorov} for modern expositions on the topic. 
The formal derivation assumes the dynamics to be leading order given by the linear dispersive dynamics and studies the resonant interactions between waves over long-time scales to derive a leading order nonlinear effect. 
The derivation is meant to hold for time-scales $1 \ll t \approx T_{kin}$, the so-called `kinetic time', which is a characteristic time-scale for the leading order nonlinear effects. 
The WKE for the NLS was recently given a mathematically rigorous proof by Deng and Hani in \cite{DengHani21,DengHani23,DH21_PC,DH23,deng2023long} where it was shown to correctly predict the dynamics of the deterministic nonlinear Schr\"odinger with suitable small, random initial data (with no damping or driving) in a certain parameter regime for times $t < \delta T_{kin}$, i.e. a small, fixed fraction of the kinetic time-scale (see also \cite{BGHS21} for earlier progress) and then finally to arbitrary finite times in their recent work \cite{deng2023long}.
A variety of related works have appeared recently, for example \cite{staffilani2021wave,CG20,ACG21,CHG22,DK21,DK23,dymov2021large} and the references therein. 

In this paper we are interested in discussing the statistically stationary setting, i.e. we send $t \to \infty$ first, and then send the parameters $\nu,\sigma,\lambda^{-1} \to 0$. 
Specifically, we are interested in identifying the correct scaling limits of the equations in which one should see wave turbulence (vs fully developed turbulence) and to identify easy variants of the associated constant flux balance laws.  
There are several motivations for this. Firstly, many physical systems of interest are modeled as a damped-driven, statistically stationary system (note that this setting is out of equilibrium in the statistical mechanics sense).
Indeed, most systems in the physics and engineering literature are considered in exactly this regime (see discussions in e.g. \cite{nazarenko2011wave,zakharov2012kolmogorov,ZP04}). 
Secondly, the statistically stationary setting opens up the possibility of using a variety of ideas from stochastic PDEs, ergodic theory, and random dynamical systems, which may ultimately be very helpful for making mathematically rigorous studies of wave turbulence beyond the kinetic time-scale.

In Section \ref{sec:NecSuff} we discuss necessary and sufficient conditions to characterize a full dissipation anomaly in the weak turbulence regime, in Section \ref{sec:NA} we discuss sufficient conditions to rule out any dissipation anomaly, and in Section \ref{sec:PDA} we discuss the possibility of `partial' dissipation anomalies.
In Section \ref{sec:STR} we briefly discuss the strong turbulence regime (which would be related to quantum hydrodynamics, rather than wave turbulence).
Finally in Section \ref{sec:Phys} we discuss the relationship of the results with the physics literature, such as empirical observations, and end with some mathematical conjectures. 

%\textcolor{red}{
\section{Necessary and sufficient conditions for full dissipation anomaly} \label{sec:NecSuff}

%\textcolor{red}{
By Fj{\o}rtoft's argument (see e.g. \cite{nazarenko2011wave}) suggests we will see wave action being sent by the nonlinearity to larger and larger scales (before being dissipated at a large-scale dissipation range) while Hamiltonian energy is sent to smaller and smaller scales (before being dissipated at the small-scale dissipation range).
However, in a weak turbulence regime, we expect a direct cascade of \emph{kinetic energy} rather than the full Hamiltonian energy, as the potential energy is expected to be higher order.
The (volume-averaged) wave-action flux and Hamiltonian flux through scale $N$ are defined as the following, where here $P_{\leq N}$ (resp. $P_{\geq N}$) denotes the Littlewood-Paley projection to frequencies less (resp. greater) than $N$,
%\textcolor{red}{
\begin{align*}
\Pi_{WA}[u](N) & = \mathrm{Im} \fint_{\mathbb T^d_\lambda} \overline P_{\geq N}u P_{\geq N} (\abs{u}^2 u) \dx \\
\Pi_{\mathcal{H}}[u](N) & = -\mathrm{Im} \fint_{\mathbb T^d_\lambda} \Delta \overline P_{\leq N}u \left( P_{\leq N} (\abs{u}^2 u) - \abs{P_{\leq N}u}^2 P_{\leq N}u \right) \dx \\ & \quad + \sigma \mathrm{Im} \fint_{\mathbb T^d_\lambda} P_{\leq N} (\abs{u}^2 u) \abs{P_{\leq N} u}^2 P_{\leq N} u \dx. 
\end{align*}
For most of the paper we omit the `$[u]$' as it is always clear from context.
We additionally define the kinetic energy `flux': 
\begin{align*}
\Pi_{KE}(N) & = -\mathrm{Im} \fint_{\mathbb T^d_\lambda}  \Delta P_{\leq N} \overline u P_{\leq N} (\abs{u}^2 u) \dx. 
\end{align*}
In what follows we also denote Littlewood-Paley projections via subscripts 
\begin{align*}
P_{\leq N} u = u_{\leq N}, \quad P_{\geq N} u = u_{\geq N}. 
\end{align*}
Let us briefly discuss these definitions.
Consider for a moment the conservative nonlinear Schr\"odinger equation
\begin{align}
-i \partial_t \psi = \Delta \psi + \sigma \abs{\psi}^2 \psi. \label{eq:NLSConserve}
\end{align}
As described above, the wave action is conserved, i.e. $\frac{d}{dt}\mathcal{WA}_\lambda[\psi] = 0$.
The quantity $\Pi_{WA}[\psi](N)$, satisfies
\begin{align*}
\frac{d}{dt}\mathcal{WA}_\lambda[P_{\geq N}\psi] = \sigma \Pi_{WA}[\psi](N), 
\end{align*}
and it can be interpreted as a `wave action flux through scale $N$' (in particular, a transfer of wave-action from smaller scales to larger scales).
Similarly, the Hamiltonian flux satisfies
\begin{align*}
\frac{d}{dt}\mathcal{WA}_\lambda[P_{\leq N}\psi] = \sigma \Pi_{\mathcal{H}}[\psi](N), 
\end{align*}
and so can be interpted as a `Hamiltonian flux through scale $N$' (in particular, a transfer of Hamiltonian energy from larger scales to smaller scales).
The quantity $\Pi_{KE}$ satisfies
\begin{align*}
\frac{d}{dt}\norm{\grad \psi}_{L^2_\lambda}^2 = \sigma \Pi_{KE}[\psi](N), 
\end{align*}
however, it should not be interpreted as a flux as the kinetic energy is \emph{not} a conservation law of \eqref{eq:NLSConserve}. 

If wave action is to be transferred to large scales and dissipated, then we expect that in the large-scale inertial range $N_F \ll N \ll 1$ (where $N_F = N_F(\nu)$ is a scale at which the low-frequency damping from $D$ dominates) the following holds:
\begin{align*}
\frac{1}{\sigma}\EE \Pi_{WA}(N) \approx \frac{\nu}{\sigma^2}\EE \norm{D^{1/2} u}_{L^2_\lambda}^2 \approx \eps_{wa}, 
\end{align*}
whereas in the small-scale inertial range $1 \ll N \ll N_D$ (where $N_D = N_D(\nu)$ is a scale at which the high-frequency dissipation from $D$ dominates) that the following holds: 
\begin{align*}
\frac{1}{\sigma}\EE \Pi_{\mathcal{H}}(N) \approx \frac{1}{\sigma}\Pi_{KE}(N) \approx \frac{\nu}{\sigma^2} \EE \norm{\grad D^{1/2} u}_{L^2_\lambda}^2 \approx \eps_{ke}. 
\end{align*}
Basic sufficient conditions for these identities to hold are the main content of Theorems \ref{thm:Direct} and \ref{thm:Inverse} below.

Let us briefly discuss our scaling choices before we state the theorem as it is more delicate than the analogous choices in fully developed hydrodynamic turbulence \cite{BCZPSW19,BCZPSW20,Ethan23}.
If the driving is too weak, then naturally one would not see turbulence.
For this we impose the condition
\begin{align*}
\lim_{\nu,\sigma \to 0}\frac{\nu}{\sigma^2} = 0; 
\end{align*}
to make this concept precise we take two sequences $\set{\nu_k}_{k \in \mathbb N}, \set{\sigma_k}_{k \in \mathbb N}$ and we impose
\begin{align*}
\lim_{k \to \infty} \nu_k = 0, \quad \lim_{k \to \infty} \sigma_k = 0, \quad \lim_{k \to 0}\frac{\nu_k}{\sigma^2_k} = 0. 
\end{align*}
The limit $\sigma_k \to 0$ is what implies the nonlinearity is weak.
This driving ensures that all sufficiently strong norms are unbounded, see for example \eqref{eq:WABal} above -- it is expected that sufficiently weak norms remain uniformly bounded, but this is wide open and is generally much stronger than what is necessary to prove Theorems \ref{thm:Inverse} and \ref{thm:Direct}.
Note that the fluctuation dissipation regime, $\nu \approx \sigma^2$ is not included.
It is unclear what should happen in this borderline case; however Sections \ref{sec:NA} and \ref{sec:PDA} do include this case.

Now, let us state theorems that give necessary and sufficient conditions for cascade flux balance laws corresponding to a ``full'' dissipation anomaly under a condition of strong driving and weak nonlinearity.
%\textcolor{red}{
See \cite{BCZPSW20,Ethan23} for analogues of this result for the stochastically-forced 2D Navier-Stokes and \cite{BCZPSW19,Ethan23} for 3D Navier-Stokes.
Note in those equations there is only a strong turbulence regime and so the conditions are simpler to interpret in that case.
Furthemore, note that the case of 2D NSE is more analogous to the case of \eqref{def:NLS} as it has a dual cascade with two exact quadratic conservation laws. 

The first result shows that the most natural condition -- the vanishing of wave-action dissipation at finite scales -- is equivalent to the existence of an inertial range over which the nonlinear wave action flux exactly balances the input wave-action.
\begin{theorem}[Inverse cascade] \label{thm:Inverse}
Let $\mu_{\lambda,\sigma,\nu}$ be a family of invariant measures of the NLS \eqref{def:NLS} for a sequence of parameters $(\lambda_k,\sigma_k,\nu_k) \to (\infty,0,0)$. Suppose that the external forcing satisfies Assumption \ref{Ass:g}. 
Suppose that the strong driving condition holds:
\begin{align*}
\lim_{k \to \infty}\frac{\nu_k}{\sigma^2_k} = 0. 
\end{align*}
Then the following are equivalent:
\begin{itemize}
\item[(i)] $\forall N_0 > 0$ fixed there holds
\begin{align}
\lim_{k \to \infty}\frac{\nu_k}{\sigma^2_k} \EE \norm{D^{1/2} P_{\geq N_0} u}_{L^2_{\lambda_k}}^2 = 0. \label{eq:ADWA}
\end{align}
\item[(ii)] For all $k \geq 1$, $\exists N_{F}= N_F(k)$ with $\lim_{k \to \infty} N_F = 0$ such that
\begin{align}
\lim_{N_0 \to 0} \limsup_{k \to \infty} \sup_{N \in (N_F,N_0)} \abs{\frac{1}{\sigma_k}\EE \Pi_{WA}(N) - \eps_{wa}} = 0; \label{eq:PiWA}
\end{align}
\end{itemize}
\end{theorem}

The case of the direct cascade of kinetic energy is more interesting, as some weak nonlinearity assumptions are required for the flux of the Hamiltonian energy and kinetic energy to be asymptotically equal (and balance the input $\eps_{ke}$).  

\begin{theorem}[Direct cascade of kinetic energy]\label{thm:Direct}
Let $\mu_{\lambda,\sigma,\nu}$ be a family of invariant measures of the NLS \eqref{def:NLS} for a sequence of parameters $(\lambda_k,\sigma_k,\nu_k) \to (\infty,0,0)$. Suppose that the external forcing satisfies Assumption \ref{Ass:g}. 
Suppose that the following conditions hold:
\begin{itemize}
\item[(i)] sufficiently strong driving:
\begin{align*}
\lim_{k \to \infty}\frac{\nu_k}{\sigma^2_k} = 0;
\end{align*}
\item[(ii)] weak nonlinearity and vanishing potential energy dissipation: $\forall N_0 > 0$ fixed, there holds 
\begin{align}
\lim_{k \to \infty} \sigma_k \EE \norm{u}_{L^4_{\lambda_k}}^2 & = 0 \label{eq:L2van} \\
\lim_{k\to \infty}\frac{\nu_k}{\sigma_k} \EE \norm{P_{\leq N_0} u}_{L^4_{\lambda_k}}^4 & = 0.  \label{eq:PEDvan}
\end{align}
Then, the following are equivalent
\begin{itemize}
\item[(i)] $\forall N_0 > 0$ fixed there holds
\begin{align}
\lim_{k\to \infty}\frac{\nu_k}{\sigma^2_k} \EE \norm{\grad D^{1/2} P_{\leq N_0} u}_{L^2_{\lambda_k}}^2 & = 0.  \label{eq:ADKE}
\end{align}
\item[(ii)] $\forall k \geq 1$, $\exists N_D(k) $ such that $N_D \to \infty$ as $k \to \infty$ such that 
\begin{align}
\lim_{N_0 \to \infty} \limsup_{k \to \infty} \sup_{N \in (N_0,N_D)} \abs{\frac{1}{\sigma_k}\EE \Pi_{KE}(N) - \eps_{ke}} = 0; \label{eq:PiKE}
\end{align}
and
\begin{align}
\lim_{N_0 \to \infty} \limsup_{k \to \infty} \sup_{N \in (N_0,N_D)} \abs{\frac{1}{\sigma}\EE \Pi_{\mathcal{H}}(N) - \eps_{ke}} = 0. \label{eq:PiH}
\end{align}
\end{itemize} 
\end{itemize} 
\end{theorem}

\begin{remark}
That we take $N_0 \to 0$ in Theorem \ref{thm:Inverse} and $N_0 \to \infty$ in Theorem \ref{thm:Direct} is remove any effects of details of the forcing profies $g^\lambda_j$. For example, they are not assumed to be compactly supported in frequency. 
\end{remark}

Theorem \ref{thm:Direct} does not seem to directly imply the expected balance of kinetic energy dissipation
\begin{align*}
\frac{\nu}{\sigma^2} \EE \norm{\grad D^{1/2} u}_{L^2_\lambda}^2 \approx \eps_{ke}. 
\end{align*}
However, with a slight strengthening of \eqref{eq:PEDvan} we have the following.
\begin{lemma} \label{lem:KEbalance}
Let $\mu_{\lambda,\sigma,\nu}$ be a family of invariant measures of the NLS \eqref{def:NLS} for a sequence of parameters $(\lambda_k,\sigma_k,\nu_k) \to (\infty,0,0)$. Suppose that the external forcing satisfies Assumption \ref{Ass:g}.
If we assume \eqref{eq:L2van} together with 
\begin{align}
\lim_{k \to \infty} \frac{\nu_k}{\sigma_k} \EE \norm{u}_{L^4_{\lambda_k}}^4 & = 0, \label{eq:L4van} \\
\lim_{k \to \infty} \frac{\nu_k}{\sigma_k} \EE \norm{u}_{L^4_{\lambda_k}}^2\norm{\grad u}_{L^4_{\lambda_k}}^2 & = 0, \label{eq:PEkiller}
\end{align}
then there holds 
\begin{align*}
\lim_{k \to \infty} \frac{\nu_k}{\sigma^2_k} \EE \norm{\grad D^{1/2} u}_{L^2_{\lambda_k}}^2 = \eps_{ke}.  
\end{align*}
\end{lemma}
All of the results in this note are some variation of the following basic flux balances, which can be simply interpreted as saying that the nonlinearity must dissipate whatever the dissipation does not.
The proof of this proposition is evident from the proofs of Theorems \ref{thm:Inverse} and \ref{thm:Direct}. 
\begin{proposition}
Let $\mu_{\lambda,\sigma,\nu}$ be a family of invariant measures of the NLS \eqref{def:NLS} for a sequence of parameters $(\lambda_k,\sigma_k,\nu_k) \to (\infty,0,0)$. Suppose that the external forcing satisfies Assumption \ref{Ass:g}. 
Then $\forall N \leq 1$ there holds
\begin{align*}
\limsup_{k \to \infty} \abs{\frac{1}{\sigma_k}\Pi_{WA}(N) - \frac{\nu_k}{\sigma_k^2} \EE \norm{D^{1/2} u_{< N}}_{L^2_{\lambda_k}}^2 } = o_{N \to 0}(1). 
\end{align*}
If we further assume the weak nonlinearity assumptions \eqref{eq:L2van}, \eqref{eq:L4van}, and \eqref{eq:PEkiller}, then $\forall N \geq 1$ there holds
\begin{align*}
\limsup_{k \to \infty} \abs{\frac{1}{\sigma_k}\Pi_{KE}(N) - \frac{\nu_k}{\sigma_k^2} \EE \norm{\grad D^{1/2} u_{> N}}_{L^2_{\lambda_k}}^2 } = o_{N \to \infty}(1), 
\end{align*}
and
\begin{align*}
\limsup_{k \to \infty} \abs{\frac{1}{\sigma_k}\Pi_{\mathcal{H}}(N) - \frac{\nu_k}{\sigma_k^2} \EE \norm{\grad D^{1/2} u_{> N}}_{L^2_{\lambda_k}}^2 } = o_{N \to \infty}(1).
\end{align*}
\end{proposition}

\begin{remark}
%\textcolor{red}{
Note that the following a priori estimates 
\begin{align}
\lim_{k \to \infty}\frac{\nu_k}{\sigma^2_k} \EE \norm{D^{1/2} P_{\geq 1} u}_{L^2_{\lambda_k}}^2 & = 0  \label{eq:ADWAeasy} \\ 
\lim_{k \to \infty}\frac{\nu_k}{\sigma^2_k} \EE \norm{\grad D^{1/2} P_{\leq 1} u}_{L^2_{\lambda_k}}^2 & = 0,  \label{eq:ADKEeasy}
\end{align}
together imply \eqref{eq:ADWA} and \eqref{eq:ADKE}.
\end{remark}

\begin{remark}
%\textcolor{red}{
Condition (ii) in Theorem \ref{thm:Direct} can be interpreted as the the vanishing potential energy input and dissipation.
The proof will clarify that the latter two conditions are claiming an inertial and integral range over which the potential energy dissipation vanishes whereas the first condition implies that the potential energy input will vanish (and hence the Hamiltonian energy is primarily dominated by the kinetic energy). 
\end{remark}

\subsection{Preliminaries}
For the forcing, we make the following frequency localization assumptions.
\begin{assumption} \label{Ass:g}
We assume that the family of forcing profiles $\set{g_j^\lambda}_{\lambda \geq 1, j \geq 1}$ satisfy the following: 
\begin{align}
&\frac{1}{2}\sum_j \norm{g_j^\lambda}_{L^2_\lambda}^2 := \eps_{wa} \approx 1 \\
&\frac{1}{2}\sum_j \norm{\grad g_j^\lambda}_{L^2_\lambda}^2 := \eps_{ke} \approx 1 \\
&\sup_{\lambda \in (1,\infty)} \sum_j\left(\norm{g_j^\lambda}_{L^4_\lambda}^2 + \norm{\Delta g_j^\lambda}_{L^2_\lambda}^2\right) <\infty \label{ineq:gL4}\\ 
&\lim_{N \to 0}  \sup_{\lambda \in (1,\infty)} \sum_j \norm{ P_{\leq N} g_j^\lambda}_{L^2_{\lambda}}^2 = 0 \label{ineq:glofreq} \\
&\lim_{N \to \infty}  \sup_{\lambda \in (1,\infty)} \sum_j \norm{ \grad P_{\geq N} g_j^\lambda}_{L^2_{\lambda}}^2= 0, \label{ineq:ghifreq}
\end{align}
\end{assumption} 
Roughly speaking, the latter three assumptions ensure that the forcing is concentrated at frequencies $\approx 1$.
We did not attempt to find the weakest possible conditions.
\begin{example}[Examples of admissible forcing]
It is simpler to construct examples of this forcing if we restrict to $\lambda \in \mathbb N$ and parameterize $\mathbb T^2$ as $[0,2\pi\lambda)^2$. 
For example one can choose the following frequency-localized forcing
\begin{align*}
\sum_{j \in \mathbb Z^d: 1 \leq \abs{j} \leq J_0}  q_j e^{i j\cdot x } dW^{(j)}_t, 
\end{align*}
for some fixed constants $q_j \in \mathbb C$ and $J_0 \geq 1$. In this case, we do not need to index the Brownian motions by $\lambda$.
This forcing satisfies Assumption \ref{Ass:g}, however, it will almost certainly not lead to a unique invariant measure in general. This is due to the fact that the forcing is $2\pi$-periodic in space while the domain is $2\pi \lambda$ periodic.
It may be possible to prove that there is a unique invariant measure with non-trivial support in all Fourier modes, however, no result of this type has appeared in the literature for any similar problem at the time of writing. 
With this in mind, it would be more natural to choose a forcing which does not have any long-range correlations. 
One simple and explicit way to do this is by the following localization method; we show the construction in $d=2$ for simplicity but one can do analogously in any dimension. 
Denote for $j \in \mathbb Z^2$ with $0 \leq j_\ell \leq \lambda-1$ the cube
\begin{align*}
Q_{j} = [\frac{j_1}{2\pi},\frac{j_1+1}{2\pi}) \times [\frac{j_2}{2\pi},\frac{j_2 + 1}{2\pi}).
\end{align*}
Denote $x_j$ the center of the cube
\begin{align*}
x_j := (\frac{j_1 + 1/2}{2\pi}, \frac{j_2 + 1/2}{2\pi}). 
\end{align*}
Let $\chi$ be a radially symmetric, smooth cutoff which satisfies  $\chi \equiv 1$ on $Q_{(1,1)}$ and vanishes outside of the support of $\cup Q_{1\pm 1,1\pm1}$. 
Then, consider forcing of the form:  
\begin{align*}
\sum_{j \in \mathbb Z^2: 0 < \abs{j}_{\ell^\infty} \leq \lambda - 1} \quad \sum_{\ell \in \mathbb Z^2: 0 < \abs{\ell} \leq \ell_0 }q_{\ell} \chi \left(x - j \right) e^{i \ell \cdot (x-x_j) } dW_t^{(j,\ell)},
\end{align*}
where note that we are indexing the Brownian motions over both $j$ and $\ell$. Here $\ell_0$ is fixed finite; if one  wants to consider forcing which is rougher in space, one could take $\ell_0 = \infty$ and impose a decay condition on $q_\ell$ to obtain whichever desired regularity class.  The forcing profiles would be $g_{(j,\ell)}^\lambda = q_{\ell} \chi \left(x - j \right) e^{i \ell \cdot (x-x_j) }$. While the Brownian motions and profiles do not depend on $\lambda$ per se, the range that the first index, $j$, runs over is $\mathbb Z^2$ with $ 0 < \abs{j}_{\ell^\infty} \leq \lambda-1$, and so the form of the forcing does vary in $\lambda$.
While there are no existing results of this type in the literature, we believe it might be possible to prove that there is a unique invariant measure for all $\lambda$ if one chooses this forcing (assuming $\ell_0 > 2$) suitably non-degenerate; for example, a variant of \cite{HM11} suitably adapted to the energy structure of nonlinear Schr\"odinger may be able to do this.  %(see also the discussion in Remark \ref{rmk:hypo}).}
\end{example}

It is classical to obtain almost-sure global well-posedness of the SPDE defined by \eqref{def:NLS}.
\begin{definition} \label{def:MildFormNLS}
We say an $\mathcal{F}_t$-adapted Markov process $u(t) \in C([0,T];H^1) \cap L^2(0,T;H^2)$ for some $0 < T < \infty$ solves \eqref{def:NLS} in the mild form if
\begin{align*}
u(t) = e^{-t(\nu D + i\Delta)}u_0 + \Gamma_t + i\sigma\int_0^t e^{-(t-s)(\nu D + i\Delta)}\abs{u}^2 u(s) \dee s, 
\end{align*}
where $\Gamma_t$ is given by the stochastic convolution 
\begin{align*}
\Gamma_t = \sigma\sum_j \int_0^t e^{-(t-\tau)(\nu D + i\Delta)}g_j^\lambda \dee W_\tau^{(j)}.  
\end{align*}
\end{definition}

The following theorem is by-now classical (see e.g. \cite{KS04} and the references therein). 
\begin{theorem}[See e.g. \cite{KS04}] \label{lem:NLSGWP}
Suppose $\lambda,\nu,\sigma> 0$ are fixed. 
Let $u_0$ be a $\mathcal{F}_0$-measurable random variable with values in $H^1$ independent from $\set{W_t^{(j,\lambda)}}$.
Then, for $\PP$-almost every $\omega \in \Omega$, there exists a unique $\mathcal{F}_t$-adapted Markov process $u(t)$ which satisfies $\forall T <\infty$, $u \in C([0,T];H^1) \cap L^2(0,T;H^2)$ and which solves \eqref{def:NLS} in the sense of Definition \ref{def:MildFormNLS}. 
Furthermore, the solutions are associated with a Feller Markov semigroup $\mathcal{P}_t$ and there exists at least one stationary measure $\mu$  (i.e. $\mathcal{P}_t^\ast \mu = \mu$) which is supported on $C^\infty$. 
\end{theorem} 
%\begin{remark} \label{rmk:hypo}
%  \textcolor{red}{ It might be possible to adapt existing results to prove uniqueness of the stationary measure in $H^1$ under very mild assumptions on the forcing using the methods of \cite{HM11}. The result [Theorem 8.21; \cite{HM11}] covers the case of real Ginzburg-Landau, however this case has a significantly different energy structure due to the nonlinear damping in the assumptions. 

%    However, we believe the adaptation to the complex cubic Ginzburg-Landau is possible in $d=2,3$, replacing the estimates available from Assumption RD.1 of \cite{HM11} with drift condition estimates on the wave action and Hamiltonian similar to those in \cite{KS04} combined with suitable parabolic regularity estimates.
%\end{remark}

\subsection{Proof of Theorem \ref{thm:Inverse}}
We begin with the proof of Theorem \ref{thm:Inverse}, which is rather straightforward after the following lemma of diagonalization type.  
\begin{lemma} \label{lem:diag}
Let $\set{A_{N,k}}_{(N,k) \in \mathbb N^2}$ be a non-negative, two-parameter sequence.  
If $\forall N$, $\lim_{k \to \infty} A_{N,k} = 0$, then $\exists$ a non-decreasing sequence ${N_k}$ with $\lim_{k \to \infty} N_k = \infty$ such that $\lim_{k \to \infty} A_{N_k,k} = 0$.
\end{lemma}
\begin{proof}
By definition, $\forall N, \exists k_N$ such that for $k \geq k_N$, $A_{N,k} < 2^{-N}$.
Without loss of generality we may take $k_N$ to be strictly increasing by replacing $k_N \mapsto \max(k_1,..,k_{N-1},k_N) + 1$.  
We now define the sequence $\set{N_k}_{k \geq 1}$ as 
\begin{align*}
N_k & = 1 \quad 1 \leq k \leq k_2 \\
N_k & = 2 \quad k_2 < k \leq k_3 \\
N_k & = 3 \quad k_3 < k \leq k_4 \\
\ldots & \ldots. 
\end{align*}
The desired conclusion follows upon noting that the construction ensures that for all $k \geq k_1$, there holds, 
\begin{align*}
A_{N_k,k} \leq 2^{-N_k}. 
\end{align*}
and $\lim_{k \to \infty} N_k = \infty$ as desired. 
\end{proof} 

\begin{proof}[\textbf{Proof of Theorem \ref{thm:Inverse}}]
First we prove that \eqref{eq:ADWA} implies \eqref{eq:PiWA}. 
By Lemma \ref{lem:diag} applied to
$$A_{N,k} = \EE \frac{\nu_k}{\sigma^2_k} \EE \norm{D^{1/2} P_{\geq N^{-1}} u}_{L^2_{\lambda_k}}^2 = 0,$$
we see that \eqref{eq:ADWA} implies $\exists N_F(k) \to 0$ such that
\begin{align}
\lim_{k \to \infty}\frac{\nu_k}{\sigma^2_k} \EE \norm{D^{1/2} P_{\geq N_F} u}_{L^2_{\lambda_k}}^2 = 0. \label{def:NF}
\end{align}
In the proof here and in all proofs in the sequel, we omit the $k$ subscripts on $\nu,\sigma$, and $\lambda$ for simplicity of notation.

By projecting \eqref{def:NLS} with $P_{\geq N}$, pairing with $\overline{u}$, and applying It\^o's lemma we have for any $N > 0$, 
\begin{align}
\frac{1}{\sigma}\EE \Pi_{WA}(N) = - \frac{\nu}{\sigma^2} \EE \norm{D^{1/2} P_{\geq N} u}_{L^2_\lambda}^2 + \frac{1}{2} \sum_{j} \fint \abs{P_{\geq N} g_{j}^\lambda}^2 \dx.  \label{eq:WAFB}
\end{align}
First consider the external forcing term.
By the definition of $\eps_{wa}$ we have
\begin{align*}
\abs{\frac{1}{2} \sum_{j} \fint \abs{P_{\geq N} g_{j}^\lambda}^2 \dx - \eps_{wa}} & \lesssim \sum_{j} \norm{P_{\lesssim N} g_j^L}_{L^2_\lambda}^2.
\end{align*}
Therefore, the contribution of the forcing converges to $\eps_{wa}$ as $N \to 0$ uniformly in $\lambda$ by the low-frequency assumption \eqref{ineq:glofreq} on the noise. 
Putting this together with \eqref{def:NF} implies the desired result \eqref{eq:PiWA}.

Next, we prove that \eqref{eq:PiWA} implies \eqref{eq:ADWA}.
By the above argument we have
\begin{align*}
\frac{1}{\sigma}\EE \Pi_{WA}(N) - \eps_{wa} = - \frac{\nu}{\sigma^2} \EE \norm{D^{1/2} P_{\geq N} u}_{L^2_\lambda}^2 + o_{N \to 0}(1).
\end{align*}
Hence, by the assumption \eqref{eq:PiWA},  $\exists N_F \to 0$ such that
\begin{align*}
\limsup_{k \to \infty}\frac{\nu}{\sigma^2} \EE \norm{D^{1/2} P_{\geq N_F} u}_{L^2_\lambda}^2 & = \lim_{N_0 \to 0} \limsup_{k \to \infty} \sup_{N \in (N_F,N_0)} \frac{\nu}{\sigma^2} \EE \norm{D^{1/2} P_{\geq N} u}_{L^2_\lambda}^2 \\
& = 0, 
\end{align*}
which in particular implies \eqref{eq:ADWA} (by the monotonicity with respect to $N_0$).
This completes the proof. 
\end{proof}

\subsection{Proof of Theorem \ref{thm:Direct}}
Next we consider the more interesting case of the direct cascade.

\begin{proof}[\textbf{Proof of Theorem \ref{thm:Direct}}]
As above, we first prove that \eqref{eq:ADKE} implies \eqref{eq:PiKE} and \eqref{eq:PiH}.

By Lemma \ref{lem:diag}, \eqref{eq:ADKE} implies $\exists N_D(k) \to \infty$ such that
\begin{align*}
\lim_{\nu,\sigma \to 0}\frac{\nu}{\sigma^2} \EE \norm{\grad D^{1/2} P_{\leq N_D} u}_{L^2_\lambda}^2 & = 0.
\end{align*}
First note the `balance' of kinetic energy that arises from It\^o's lemma applied to the quantity $\frac{1}{2}\norm{\grad u_{\leq N}}_{L^2}^2$, 
\begin{align*}
\frac{1}{\sigma}\EE \Pi_{KE}(N) = -\frac{\nu}{\sigma^2} \EE \norm{\grad D^{1/2} P_{\leq N} u}_{L^2_\lambda}^2 + \frac{1}{2} \sum_{j} \fint \abs{P_{\leq N} \grad g_{\ell}^\lambda}^2 \dx.  
\end{align*}
First consider the forcing term.
We have
\begin{align*}
\abs{\frac{1}{2} \sum_{j} \fint \abs{P_{\leq N} \grad g_{j}^\lambda}^2 \dx - \eps_{ke}} & \lesssim \sum_{j} \norm{P_{\gtrsim N} \grad g_j^L}_{L^2_\lambda}^2. 
\end{align*}
Therefore, the contribution of the forcing converges to $\eps_{ke}$ as $N \to \infty$ uniformly in $\lambda$ by the uniform regularity assumption \eqref{ineq:ghifreq} on the noise. 
By Lemma \ref{lem:diag}, $\exists N_D = N_D(k)$ with $\lim_{k \to \infty} N_D = \infty$ such that
\begin{align*}
\lim_{k \to \infty} \frac{\nu_k}{\sigma^2_k} \EE \norm{\grad D^{1/2} P_{\leq N_D} u}_{L^2_{\lambda_k}}^2 = 0. 
\end{align*}
Hence we obtain \eqref{eq:PiKE}. 

Turn next to the Hamiltonian energy flux \eqref{eq:PiH}.  
Consider the balance of Hamiltonian energy arising from applying It\^o's lemma to $\mathcal{H}[P_{\leq N} u]$, yielding
\begin{align}
\frac{1}{\sigma}\EE \Pi_{\mathcal{H}}(N) & = -\frac{\nu}{\sigma^2} \EE \norm{\grad D^{1/2} P_{\leq N} u}_{L^2_\lambda}^2 - \frac{\nu}{\sigma} \EE \norm{P_{\leq N} u}_{L^4_\lambda}^4 - 3\frac{\nu}{\sigma} \EE \norm{P_{\leq N}u \grad P_{\leq N}u}_{L^2_\lambda}^2 \notag \\
& \quad + \frac{1}{2}\sum_j \fint_{\mathbb T_\lambda^d} \abs{P_{\leq N}g^\lambda_j}^2 \dx + \frac{\sigma}{2}\sum_j \fint_{\mathbb T_\lambda^d} \abs{P_{\leq N} u}^2 \abs{P_{\leq N}g^\lambda_j}^2 \dx. \label{eq:HFluxBal}
\end{align}
We have already seen that the assumptions on the noise imply
\begin{align*}
\frac{1}{2}\sum_j \fint_{\mathbb T_\lambda^d} \abs{P_{\leq N}g^\lambda_j}^2 \dx = \eps_{ke} + o_{N \to \infty}(1). 
\end{align*}
Furthermore we have (by H\"older's inequality and Bernstein's inequality (Lemma \ref{lem:Bern})), 
\begin{align*}
\frac{\sigma}{2}\sum_j \fint_{\mathbb T_\lambda^d} \abs{P_{\leq N} u}^2 \abs{P_{\leq N}g^\lambda_j}^2 \dx & \lesssim \sigma \EE \norm{P_{\leq N} u}_{L^4_\lambda}^2 \sum_j \norm{P_{\leq N} g^\lambda_j}_{L^4_\lambda}^2 \\
& \lesssim \sigma \EE \norm{u}_{L^4_\lambda}^2 \sum_j \norm{g^\lambda_j}_{L^4_\lambda}^2,
\end{align*}
which then vanishes in the limit by \eqref{eq:L2van} and the assumption \eqref{ineq:gL4} on the forcing. 
Turn next to the dissipation terms.
By Lemma \ref{lem:diag}, $\forall k$, $\exists N_D(k)$ with $\lim_{k \to \infty} N_D = \infty$ such that the following holds by \eqref{eq:L2van} and \eqref{eq:PEDvan} respectively, 
\begin{align*}
\lim_{k \to \infty} \frac{\nu}{\sigma^2} \EE \norm{\grad P_{\leq N_D} u}_{L^2_\lambda}^2 & = 0 \\
\lim_{k \to \infty} \frac{\nu}{\sigma} \EE \norm{P_{\leq N_D} u}_{L^4_\lambda}^4 & = 0. 
\end{align*}
Furthermore, by Bernstein's inequality,  
\begin{align*}
\frac{\nu}{\sigma} \EE \norm{P_{\leq N}u \grad P_{\leq N}u}_{L^2_\lambda}^2 & \lesssim \frac{\nu}{\sigma}\EE \norm{P_{\leq N}u}_{L^4}^2 \norm{\grad P_{\leq N} u}_{L^4_\lambda}^2 \\
& \lesssim N^2 \frac{\nu}{\sigma} \EE \norm{P_{\leq N}u}_{L^4_{\lambda}}^4, 
\end{align*}
and therefore, by applying Lemma \ref{lem:diag} again and adjusting $N_D$ further, we can also impose
\begin{align*}
\lim_{k \to \infty}\frac{\nu}{\sigma} \EE \norm{P_{\leq N_D}u \grad P_{\leq N_D}u}_{L^2_\lambda}^2 &  = 0, 
\end{align*}
while still having $N_D \to \infty$. 
Putting everything together we have, 
\begin{align*}
\sup_{N \in (N_0,N_D)}\abs{\frac{1}{\sigma}\EE \Pi_{\mathcal{H}}(N) - \eps_{ke}} &\lesssim \frac{\nu}{\sigma^2} \EE \norm{\grad P_{\leq N_D} u}_{L^2_\lambda}^2 +  \lim_{k \to \infty} \frac{\nu}{\sigma} \EE \norm{P_{\leq N_D} u}_{L^4_\lambda}^4 \\ 
& \quad + \frac{\nu}{\sigma} \EE \norm{P_{\leq N_D}u \grad P_{\leq N_D}u}_{L^2_\lambda}^2 + \sigma \EE \norm{u}_{L^4_\lambda}^2, + o_{N_0 \to \infty}(1), 
\end{align*}
which therefore completes the proof of \eqref{eq:PiH}.

Next, we check that \eqref{eq:PiKE} and \eqref{eq:PiH} imply \eqref{eq:ADKE}. 
This follows easily as in the proof of Theorem \ref{thm:Inverse} (and in fact, we only need to use \eqref{eq:PiKE}). 
By the balance of kinetic energy we have 
\begin{align*}
\frac{1}{\sigma}\EE \Pi_{KE}(N)  & = -\frac{\nu}{\sigma^2} \EE \norm{\grad D^{1/2} P_{\leq N} u}_{L^2_\lambda}^2 + \eps_{ke} + o_{N \to \infty}(1), 
\end{align*}
and so by \eqref{eq:PiKE}, $\exists N_D(k) \to \infty$ such that 
\begin{align*}
\limsup_{k \to \infty}\frac{\nu_k}{\sigma^2_k} \EE \norm{\grad D^{1/2} P_{\leq N_D} u}_{L^2_\lambda}^2
& \leq \lim_{N_0 \to \infty} \limsup_{k \to \infty} \sup_{N \in (N_0,N_D)} \abs{\frac{1}{\sigma}\EE \Pi_{KE}(N) - \eps_{ke}}, 
\end{align*}
from which the desired result \eqref{eq:ADKE} follows. 
\end{proof}

We end the section with a proof of Lemma \ref{lem:KEbalance}.
\begin{proof}[\textbf{Proof of Lemma \ref{lem:KEbalance}}]
  Recall the balance of Hamiltonian energy dissipation \eqref{eq:Hbal} which gives
  %\textcolor{red}{
\begin{align*}
\frac{\nu}{\sigma^2} \EE_{\mu} \norm{\grad D^{1/2} u}_{L^2_\lambda}^2 + \frac{\nu}{\sigma} \mathrm{Re} \EE \fint D \overline{u} (\abs{u}^2 u) \dx = \eps_{ke} + \frac{\sigma}{2} \EE_\mu \sum_j \fint_{\mathbb T_L^d} \abs{u}^2 \abs{g^\lambda_j}^2 \dx. 
\end{align*}
First consider the It\^o correction term. By H\"older's inequality
\begin{align*}
\frac{\sigma}{2} \EE_\mu \sum_j \fint_{\mathbb T_\lambda^d} \abs{u}^2 \abs{g^\lambda_j}^2 \dx \lesssim \sigma \EE \norm{u}_{L^4_{\lambda}}^2 \sum_{j} \norm{g^\lambda_j}_{L^4_\lambda}^2,
\end{align*}
which then vanishes in the limit by assumption \eqref{eq:L2van}. 

We next prove that  the potential energy dissipation vanishes in the limit, by observing that due to our particular choice of $D$,
\begin{align*}
\mathrm{Re} \EE \fint D \overline{u} (\abs{u}^2 u) \dx  & = \EE \norm{u}_{L^4}^4 + 3\EE \norm{u \grad u}_{L^2_\lambda}^2 \\
& \leq \EE \norm{u}_{L^4}^4 + 3\EE \norm{u}_{L^4_\lambda}^2 \norm{\grad u}_{L^4_\lambda}^2, 
\end{align*}
which then vanishes from the Hamiltonian energy balance due to \eqref{eq:L4van} \eqref{eq:PEkiller}, yielding the final result. 
\end{proof} 

\section{Characterization of non-anomaly} \label{sec:NA}
Theorems \ref{thm:Inverse} and \ref{thm:Direct} show that all of the dissipation concentrating at low frequencies (resp. high frequencies) is equivalent to the cascade flux laws.
The opposite is also true: no dissipation concentrating at low (resp. high) frequencies is equivalent to the nonlinear fluxes vanishing. 
\begin{theorem}[Vanishing wave-action cascade] \label{thm:ICFail}
Let $\mu_{\lambda,\sigma,\nu}$ be a family of invariant measures of the NLS \eqref{def:NLS} for a sequence of parameters $(\lambda_k,\sigma_k,\nu_k) \to (\infty,0,0)$.
Suppose that strong driving or fluctuation dissipation holds:
\begin{align*}
\limsup_{k \to \infty} \frac{\nu_k}{\sigma^2_k} < \infty. 
\end{align*}
Then the following are equivalent: 
\begin{itemize}
\item[(i)] Vanishing of wave-action dissipation at low frequencies
\begin{align}
\lim_{N \to 0} \limsup_{k \to \infty} \frac{\nu_k}{\sigma^2_k} \EE \norm{D^{1/2} u_{<N} }_{L^2_{\lambda_k}}^2  = 0 \label{ineq:HnsBd}
\end{align}
\item[(ii)] The vanishing of the nonlinear wave-action flux at high frequencies 
\begin{align}
\lim_{N_0 \to 0}\limsup_{k \to \infty} \sup_{N \in (0,N_0) }\abs{\frac{1}{\sigma_k}\EE \Pi_{WA}(N)} = 0. \label{eq:WAFvan}
\end{align}
\end{itemize}
\end{theorem}
\begin{remark}
Notice that any a priori estimate on the low frequencies of the form: $\exists \delta > 0$ such that: 
\begin{align*}
\limsup_{\nu,\sigma \to 0} \frac{\nu}{\sigma^2}\norm{\abs{\grad}^{-\delta} D^{1/2} u_{\leq 1}}_{L^2_\lambda}^2 < \infty,
\end{align*}
is sufficient to prove \eqref{ineq:HnsBd}. 
\end{remark}

The corresponding theorem for the kinetic energy cascade is the following.

\begin{theorem}[Vanishing kinetic energy cascade] \label{thm:DCFail}
Let $\mu_{\lambda,\sigma,\nu}$ be a family of invariant measures of the NLS \eqref{def:NLS} for a sequence of parameters $(\lambda_k,\sigma_k,\nu_k) \to (\infty,0,0)$.
Suppose that the following conditions hold:
\begin{itemize}
\item[(i)] strong driving or fluctuation dissipation:
\begin{align*}
\limsup_{k \to \infty} \frac{\nu_k}{\sigma^2_k} < \infty. 
\end{align*}
\item[(ii)] the weak nonlinearity assumptions \eqref{eq:L2van}, \eqref{eq:L4van}, and \eqref{eq:PEkiller}.  
\end{itemize} 
Then, the following are equivalent: 
\begin{itemize}
\item[(i)] The vanishing of kinetic energy dissipation at high frequencies 
\begin{align}
\lim_{N \to \infty} \limsup_{k \to \infty} \frac{\nu_k}{\sigma^2_k} \EE \norm{\grad D^{1/2} u_{> N}}_{L^2_{\lambda_k}}^2 = 0. \label{ineq:HsBdHiF}
\end{align}
\item[(ii)] The vanishing of kinetic energy and Hamiltonian energy flux at high frequencies:
\begin{align}
\lim_{N_0 \to \infty}\limsup_{k \to \infty} \sup_{N \in (N_0,\infty) }\abs{\frac{1}{\sigma_k}\EE \Pi_{\mathcal{H}}(N)} = 0 \label{eq:HFvan}
\end{align}
and
\begin{align}
\lim_{N_0 \to \infty}\limsup_{k \to \infty} \sup_{N \in (N_0,\infty) }\abs{\frac{1}{\sigma_k}\EE \Pi_{KE}(N)} = 0. \label{eq:KEvan}
\end{align}
\end{itemize} 
\end{theorem}
\begin{remark}
Notice that any a priori regularity estimate such as: $\exists \delta > 0$ such that: 
\begin{align*}
\limsup_{\nu,\sigma \to 0} \frac{\nu}{\sigma^2}\EE \norm{\abs{\grad}^{\delta} \grad D^{1/2} u_{\geq 1}}_{L^2_\lambda}^2 < \infty,
\end{align*}
is sufficient to prove \eqref{ineq:HsBdHiF}.
Hence, the formation of a true cascade requires that the a priori regularity estimates given by the dissipation balance are basically sharp, at least in terms of estimates available in $L^2$ moments in $\PP$. 
\end{remark}

\subsection{Proof of Theorems \ref{thm:ICFail} and \ref{thm:DCFail}}
\begin{proof}[\textbf{Proof of Theorem \ref{thm:ICFail}}]
As in previous proofs, we drop the $k$ subscripts. 
By the wave action dissipation balance \eqref{eq:WABal} and wave-action flux balance \eqref{eq:WAFB} we have $\forall N > 0$ (as in the proof of Theorem \ref{thm:Inverse}), 
\begin{align*}
\frac{1}{\sigma}\EE \Pi_{WA}(N) = \frac{\nu}{\sigma^2}\EE \norm{D^{1/2} P_{\leq N} u}_{L^2_\lambda}^2 + o_{N \to 0}(1), 
\end{align*}
which implies
\begin{align*}
\sup_{N \in (0,N_0)}\abs{\frac{1}{\sigma}\Pi_{WA}(N)} = \frac{\nu}{\sigma^2}\EE \norm{D^{1/2} P_{\leq N_0} u}_{L^2_\lambda}^2 + o_{N_0 \to 0}(1). 
\end{align*}
Therefore \eqref{ineq:HnsBd} is equivalent to \eqref{eq:WAFvan}. 
\end{proof}

\begin{proof}[\textbf{Proof of Theorem \ref{thm:DCFail}}]
We only consider the more interesting case of \eqref{eq:HFvan}.
By the Hamiltonian energy dissipation balance \eqref{eq:HFluxBal} we have
\begin{align*}
\frac{1}{\sigma}\EE \Pi_{\mathcal{H}}(N) & = -\frac{\nu}{\sigma^2} \EE \norm{\grad D^{1/2} P_{\leq N} u}_{L^2_\lambda}^2 - \frac{\nu}{\sigma} \EE \norm{P_{\leq N} u}_{L^4_\lambda}^4 - 3\frac{\nu}{\sigma} \EE \norm{P_{\leq N}u \grad P_{\leq N}u}_{L^2_\lambda}^2 \notag \\
& \quad + \eps_{ke} + \frac{\sigma}{2}\sum_j \fint_{\mathbb T_\lambda^d} \abs{P_{\leq N} u}^2 \abs{P_{\leq N}g^\lambda_j}^2 \dx + o_{N \to \infty}(1). 
\end{align*}
By the assumptions \eqref{eq:L2van}, \eqref{eq:L4van}, and \eqref{eq:PEkiller} as in the proof of Theorem \ref{thm:Direct} we have 
\begin{align*}
\frac{1}{\sigma}\EE \Pi_{\mathcal{H}}(N) & = \eps_{ke} - \frac{\nu}{\sigma^2} \EE \norm{\grad D^{1/2} P_{\leq N} u}_{L^2_\lambda}^2 + o_{k \to \infty}(1) + o_{N \to \infty}(1). 
\end{align*}
By Lemma \ref{lem:KEbalance}, there then holds 
\begin{align*}
\frac{1}{\sigma}\EE \Pi_{\mathcal{H}}(N) & = \frac{\nu}{\sigma^2} \EE \norm{\grad D^{1/2} P_{> N} u}_{L^2_\lambda}^2 + o_{k \to \infty}(1) + o_{N \to \infty}(1). 
\end{align*}
Therefore,
\begin{align*}
\sup_{N \geq N_0} \abs{\frac{1}{\sigma}\EE \Pi_{\mathcal{H}}(N)} & = \frac{\nu}{\sigma^2} \EE \norm{\grad D^{1/2} P_{> N_0} u}_{L^2_\lambda}^2 + o_{k \to \infty}(1) + o_{N_0 \to \infty}(1)
\end{align*}
and therefore we see that \eqref{ineq:HsBdHiF} is equivalent to \eqref{eq:HFvan}. 
\end{proof}

\section{Partial dissipation anomaly} \label{sec:PDA}
Although not clearly consistent with observations, there is a possibility that some, but not all, of the wave-action (resp. kinetic energy) will be dissipated at asymptotically large scales (resp. small scales), leaving the rest to be dissipated either in the integral or inertial range.
In the strong driving regime
\begin{align*}
\frac{\nu}{\sigma^2} \to 0 
\end{align*}
non-trivial dissipation of wave-action (or kinetic energy) at any fixed scale $N_0$ implies a divergence in amplitude: 
\begin{align}
\EE \norm{u_{\approx N_0}}_{L^{2}_\lambda}^2 \to \infty. \label{eq:diver}
\end{align}
This does not rule out the possibility that some of the wave action is sent to asymptotically low frequencies as $k \to \infty$ and, the most important condition for weak turbulence, condition \eqref{eq:L2van}, is not in contradiction to \eqref{eq:diver} either. 
We make the following assertion about nonlinear wave action flux if there is partial, but not total, dissipation anomaly.
Note that the following theorem does not require the strong driving condition, and it makes sense even if
\begin{align*}
\liminf_{k \to \infty} \frac{\nu_k}{\sigma_k^2} > 0. 
\end{align*}
\begin{theorem}[Partial inverse cascade] \label{thm:InverseP}
Let $\mu_{\lambda,\sigma,\nu}$ be a family of invariant measures of the NLS \eqref{def:NLS} for a sequence of parameters $(\lambda_k,\sigma_k,\nu_k) \to (\infty,0,0)$. Suppose that the external forcing satisfies Assumption \ref{Ass:g}. 
Let $\eps^\star_{wa} \in [0, \eps_{wa}]$ be defined by 
\begin{align}
\liminf_{N_0 \to 0} \liminf_{k \to \infty}\frac{\nu_k}{\sigma^2_k} \EE \norm{D^{1/2} P_{\leq N_0} u}_{L^2_\lambda}^2 = \eps^\star_{wa} \label{eq:ADWApartial}
\end{align}
If $\eps_{wa}^\star > 0$, then $\forall k \geq 1$, $\exists N_F(k)$ which satisfies $\lim_{k \to \infty} N_F = 0$ such that
\begin{align}
\liminf_{N_0 \to 0} \limsup_{k \to \infty} \inf_{N \in (N_F,N_0)} \frac{1}{\sigma}\EE \Pi_{WA}(N) \geq \eps_{wa}^\star. \label{eq:PiWAfail}
\end{align}
\end{theorem}
\begin{proof}
As above, by \eqref{eq:WAFB}, \eqref{eq:WABal}, and Assumption \ref{Ass:g}. we have 
\begin{align*}
\frac{1}{N}\Pi_{WA}(N) = -\frac{\nu_k}{\sigma^2_k} \EE \norm{D^{1/2} P_{\leq N} u}_{L^2_\lambda}^2 + o_{N \to 0}(1)
\end{align*}
Since for $N' \leq N$ there holds 
\begin{align*}
\frac{\nu_k}{\sigma^2_k} \EE \norm{D^{1/2} P_{\leq N'} u}_{L^2_\lambda}^2 \leq \frac{\nu_k}{\sigma^2_k} \EE \norm{D^{1/2} P_{\leq N} u}_{L^2_\lambda}^2, 
\end{align*}
for any $0 < N_F < N_0$ there holds
\begin{align*}
\inf_{N \in (N_F,N_0) } \frac{1}{N} \Pi_{WA}(N) \geq \frac{\nu_k}{\sigma^2_k} \EE \norm{D^{1/2} P_{\leq N_F} u}_{L^2_\lambda}^2 + o_{N_0 \to 0}(1). 
\end{align*}
By an argument similar to Lemma \ref{lem:diag} , the assumption \eqref{eq:ADWApartial} implies $\exists N_F(k)$ with $\lim_{k \to \infty} N_F = 0$ such that 
\begin{align*}
\liminf_{k \to \infty} \frac{\nu_k}{\sigma^2_k} \EE \norm{D^{1/2} P_{\leq N_F} u}_{L^2_\lambda}^2 = \eps^\star_{wa}. 
\end{align*}
The result then follows. 
\end{proof}

\begin{theorem}[Partial direct cascade] 
Let $\mu_{\lambda,\sigma,\nu}$ be a family of invariant measures of the NLS \eqref{def:NLS} for a sequence of parameters $(\lambda_k,\sigma_k,\nu_k) \to (\infty,0,0)$. Suppose that the external forcing satisfies Assumption \ref{Ass:g}. 
Suppose that the following conditions hold:
\begin{align}
\lim_{k \to \infty} \sigma \EE \norm{u}_{L^4_\lambda}^2 & = 0, \label{eq:L2vanP} \\
\lim_{k \to \infty} \frac{\nu}{\sigma} \EE \norm{u}_{L^4_\lambda}^4 & = 0, \label{eq:L4vanP} \\ 
\lim_{k \to \infty} \frac{\nu}{\sigma} \EE \norm{u}_{L^4_\lambda}^2\norm{\grad u}_{L^4_\lambda}^2 & = 0. \label{eq:PEkillerP}
\end{align}
Let $\eps^\star_{ke} \in [0, \eps_{ke}]$ be defined by 
\begin{align}
\liminf_{N_0 \to \infty} \liminf_{k \to \infty}\frac{\nu_k}{\sigma^2_k} \EE \norm{\grad D^{1/2} P_{> N_0} u}_{L^2_\lambda}^2 = \eps^\star_{ke} \label{eq:ADKEpartial}
\end{align}
If $\eps^{\star}_{ke} > 0$, then  $\forall k \geq 1$, $\exists N_D(k) \to \infty$ as $k \to \infty$ such that 
\begin{align}
\liminf_{N_0 \to \infty} \liminf_{k \to \infty} \inf_{N \in (N_0,N_D)} \frac{1}{\sigma}\EE \Pi_{KE}(N) \geq \eps_{ke}^\star; \label{eq:PiKEPartial}
\end{align}
and
\begin{align}
\liminf_{N_0 \to \infty} \limsup_{k \to \infty } \inf_{N \in (N_0,N_D)} \frac{1}{\sigma}\EE \Pi_{\mathcal{H}}(N) \geq \eps_{ke}^\star . \label{eq:PiHPartial}
\end{align}
\end{theorem}
\begin{proof}
The proof proceeds by combining the arguments of Theorem \ref{thm:InverseP} with Theorem \ref{thm:Direct}, so we only provide a sketch.
We will only consider \eqref{eq:PiHPartial} as \eqref{eq:PiKEPartial} is the same but easier.
By Lemma \ref{lem:KEbalance} and \eqref{eq:HFluxBal} and the arguments applied in the proof of Theorem \ref{thm:Direct}, there holds 
\begin{align*}
\frac{1}{\sigma}\EE \Pi_{\mathcal{H}}(N) & = \frac{\nu}{\sigma^2} \EE \norm{\grad D^{1/2} P_{> N} u}_{L^2_\lambda}^2 - \frac{\nu}{\sigma} \EE \norm{P_{\leq N} u}_{L^4_\lambda}^4 - 3\frac{\nu}{\sigma} \EE \norm{P_{\leq N}u \grad P_{\leq N}u}_{L^2_\lambda}^2 \notag \\
& \quad + \frac{\sigma}{2}\sum_j \fint_{\mathbb T_\lambda^d} \abs{P_{\leq N} u}^2 \abs{P_{\leq N}g^\lambda_j}^2 \dx + o_{N\to \infty}(1). 
\end{align*}
By the assumption \eqref{eq:L2vanP} there holds (as above)
\begin{align*}
\limsup_{k \to \infty} \sup_{N \geq 1} \frac{\sigma}{2}\sum_j \fint_{\mathbb T_\lambda^d} \abs{P_{\leq N} u}^2 \abs{P_{\leq N}g^\lambda_j}^2 \dx = 0 
\end{align*}
and by assumptions \eqref{eq:L4vanP} and \eqref{eq:PEkillerP},
\begin{align*}
\limsup_{k \to \infty} \sup_{N \geq 1} \frac{\nu_k}{\sigma_k} \EE \norm{P_{\leq N} u}_{L^4_\lambda}^4 & = 0 \\ 
\limsup_{k \to \infty} \sup_{N \geq 1} \frac{\nu_k}{\sigma_k} \EE \norm{P_{\leq N}u \grad P_{\leq N}u}_{L^2_\lambda}^2 & = 0. 
\end{align*}
As above, by an argument similar to Lemma \ref{lem:diag} , the assumption \eqref{eq:ADKEpartial} implies $\exists N_D(k)$ with $\lim_{k \to \infty} N_D = \infty$ such that
\begin{align*}
\liminf_{k \to \infty} \frac{\nu_k}{\sigma^2_k} \EE \norm{D^{1/2} P_{\geq N_D} u}_{L^2_\lambda}^2 = \eps^\star_{ke}. 
\end{align*}
Putting these observations together we have, 
\begin{align*}
\inf_{N \in (N_0,N_D)} \frac{1}{\sigma}\EE \Pi_{\mathcal{H}}(N) \geq \frac{\nu}{\sigma^2} \EE \norm{\grad D^{1/2} P_{> N_D} u}_{L^2_\lambda}^2 + o_{N_0 \to \infty}(1) + o_{k \to \infty}(1),  
\end{align*}
from which the desired result follows. 
\end{proof} 

\section{Remarks on the strong turbulence regime} \label{sec:STR}

In the `strong' turbulence regime, it is expected that the NLS behaves more like quantum hydrodynamic turbulence, characterized mostly as a tangle of interacting and reconnecting vortex filaments (see e.g. \cite{Tsubota08,Tsubota2017,PL11,LP10} and the references therein), rather than being dispersive wave dominated in any traditional sense.
We end with a remark which follows up the idea that \eqref{eq:L2van} is the only assumption we are making which seems to really distinguish between `weak' and `strong' turbulence at the level of the cascade flux laws.
In particular, the most fundamental marker of weak vs strong turbulence is simply whether or not the direct cascade is dominated by the `linear' kinetic energy or the `fully nonlinear' Hamiltonian energy; see for example, discussions in \cite{ZP04}. 
Without an assumption like \eqref{eq:L2van}, one can instead prove a cascade flux law for the Hamiltonian energy, rather than the kinetic energy.
The assumptions that replace \eqref{eq:L2van} are basically (A) assuming that the Hamiltonian energy input has a well-defined limit, at least on a subsequence of parameters  (assumption \eqref{eq:eH}) and (B) an additional regularity estimate that implies that the potential energy is mostly being injected at large scales (assumption \eqref{eq:LSPE}). 
The most clear indicator of a `strong' turbulence regime is then that $\eps_{\mathcal{H}} > \eps_{ke}$, although perhaps there might be intermediate regimes with $\sigma \to 0$ that have some mixture of both wave and strong turbulence. 

\begin{theorem} \label{rmk:Strong}
Let $\mu_{\lambda,\sigma,\nu}$ be a family of invariant measures of the NLS \eqref{def:NLS} for a sequence of parameters $(\lambda_k,\sigma_k,\nu_k)$ with $\nu_k \to 0$ and where we assume strong driving (however, note that we do not necessarily require $\sigma \to 0$ or even $\lambda \to \infty$)
\begin{align*}
\lim_{k \to \infty} \frac{\nu_k}{\sigma^2_k} = 0. 
\end{align*}
Suppose that there exists $\exists \eps_{\mathcal{H}} < \infty$ such that
\begin{align}
\eps_{\mathcal{H}} = \eps_{ke} + \lim_{k \to \infty} \frac{\sigma_k}{2}\EE \sum_j \fint_{\mathbb T_{\lambda_k}^d} \abs{u}^2 \abs{g^{\lambda_k}_j}^2 \dx. \label{eq:eH}
\end{align}
Suppose that for all $N_0$ there holds the following vanishing dissipation conditions: 
\begin{align}
\lim_{k\to \infty}\frac{\nu_k}{\sigma^2_k} \EE \norm{\grad D^{1/2} P_{\leq N_0} u}_{L^2_{\lambda_k}}^2 & = 0,  \label{eq:ADKEI} \\ 
\lim_{k \to \infty} \frac{\nu_k}{\sigma_k} \EE \norm{P_{\leq N_0} u }_{L^4_{\lambda_k}}^4 & = 0.  \label{ineq:PEKillerWkI}
\end{align}
Suppose that the following regularity assumptions hold:
\begin{align}
\limsup_{k \to \infty} \sigma_k \EE \norm{u}_{L^4_{\lambda_k}}^2  < \infty, \label{eq:NotExplode} \\ 
\lim_{N \to \infty }\limsup_{k \to \infty} \sigma_k \EE \norm{u}_{L^4_{\lambda_k}} \norm{u_{>N}}_{L^4_{\lambda_k}}  = 0, \label{eq:LSPE}
\end{align}
Finally, assume the following mild additional regularity estimate on the noise
\begin{align}
\limsup_{N \to \infty}  \sup_{\lambda \in (1,\infty)} \sum_j \norm{P_{\geq N} g_j^\lambda}_{L^4_{\lambda_k}}^2 = 0. \label{ineq:ghifreqL4}
\end{align}
Then $\exists N_{D} = N_D(k)$ with $\lim_{k \to \infty} N_D = \infty$ such that
\begin{align}
\lim_{N_0 \to \infty} \limsup_{k \to \infty} \sup_{N \in (N_0,N_D)} \abs{\frac{1}{\sigma_k}\Pi_{\mathcal{H}}(N) - \eps_{\mathcal{H}}} = 0. \label{eq:PiHST}
\end{align}
\end{theorem}

\begin{remark} \label{rmk:Apriori}
A priori estimates such as the following for any $J < \infty$,  
\begin{align}
\limsup_{k\to \infty}\frac{\nu_k}{\sigma^2_k} \left(\EE \norm{\grad P_{\leq 1}u}_{L^2_{\lambda_k}}^2 + \EE \norm{\brak{\grad}^{-J} P_{\geq 1}u}_{L^2_{\lambda_k}}^2\right) & = 0  \label{ineq:WADI} \\ 
\limsup_{k \to \infty} \frac{\nu_k}{\sigma_k} \left(\EE \norm{P_{\leq 1}u}_{L^4_{\lambda_k}}^4 + \frac{1}{\lambda_k^{d}}\EE \norm{\brak{\grad}^{-J} P_{\geq 1}u}_{L^2_{\lambda_k}}^4\right) & = 0, \label{ineq:WADII}
\end{align}
imply \eqref{eq:ADKEI} and \eqref{ineq:PEKillerWkI}. See below for a proof. 
\end{remark}

\begin{proof}[\textbf{Proof of Theorem \ref{rmk:Strong}}]
As previously, we omit the subscripts on the parameters. 
By the Hamiltonian flux balance \eqref{eq:HFluxBal}, there holds 
\begin{align*}
\frac{1}{\sigma} \Pi_{\mathcal{H}} & = -\frac{\nu}{\sigma^2} \EE\norm{\grad D^{1/2} u_{<N}}_{L^2_\lambda}^2 \\
& \quad - \frac{\nu}{\sigma}\EE \norm{u_{<N}}_{L^4_\lambda}^4  - 3\frac{\nu}{\sigma}\EE \norm{u_{<N} \grad u_{< N}}_{ L^2_\lambda}^2  \\ 
& \quad + \frac{1}{2}\sum_j \fint \abs{P_{<N} \grad g^\lambda_j}^2 \dx \\
& \quad + \frac{1}{2}\sum_j \EE \fint \abs{P_{<N} u}^2 \abs{P_{<N} g^\lambda_j}^2 \dx. 
\end{align*}
As in the proof of Theorem \ref{thm:Direct} there holds 
\begin{align*}
\sum_j \fint \abs{P_{<N} \grad g^\lambda_j}^2 \dx = \eps_{ke} + o_{N \to \infty}(1). 
\end{align*}
Moreover, the It\^o correction converges by \eqref{eq:NotExplode}, \eqref{eq:LSPE}, and \eqref{ineq:ghifreqL4}
\begin{align*}
\sigma \abs{\sum_j \EE \fint \abs{P_{<N} u}^2 \abs{P_{<N} g^\lambda_j}^2 \dx - \sum_j \EE \fint \abs{u}^2 \abs{g^\lambda_j}^2 \dx} & \lesssim  \sigma \EE \norm{u}_{L^4_\lambda} \norm{u_{>N}}_{L^4_\lambda} \sum_j \norm{g^\lambda_j}_{L^4_\lambda}^2 \\
& \quad + \sigma \EE \norm{u}_{L^4_\lambda}^2 \sum_j \norm{g^\lambda_j}_{L^4_\lambda} \norm{P_{>N} g^\lambda_j}_{L^4_\lambda}. 
\end{align*}
Therefore it follows that 
\begin{align*}
\sum_j \fint \abs{P_{<N} \grad g^\lambda_j}^2 \dx + \sum_j \EE \fint \abs{P_{<N} u}^2 \abs{P_{<N} g^\lambda_j}^2 \dx = \eps_{\mathcal{H}} + o_{k \to \infty, N \to \infty}(1). 
\end{align*}
By Lemma \ref{lem:diag} we may find a sequence $N_D(k) \to \infty$ such that all of the dissipation terms (both kinetic and potential) vanish in the $k \to \infty$ limit over the inertial range $N_0 \leq N \leq N_D$.
Hence, as in the proof of Theorem \ref{thm:Direct}, the desired result follows. 
\end{proof}

\begin{proof}[\textbf{Proof of Remark \ref{rmk:Apriori}}]
By Bernstein's inequalities (see Lemma \ref{lem:Bern}) we have 
\begin{align*}
\frac{\nu_k}{\sigma^2_k} \EE \norm{\grad D^{1/2} P_{\leq N_0} u}_{L^2_{\lambda_k}}^2 & \lesssim N_0^2 \frac{\nu_k}{\sigma^2_k} \EE \norm{\grad P_{\leq N_0} u}_{L^2_{\lambda_k}}^2 \\
& \lesssim N_0^2 \frac{\nu_k}{\sigma^2_k} \left(\EE \norm{\grad P_{\leq 1} u}_{L^2_{\lambda_k}}^2 + N_0^{2+J}\EE \norm{\brak{\grad}^{-J} P_{> 1} u}_{L^2_{\lambda_k}}^2\right), 
\end{align*}
which shows that \eqref{ineq:WADI} implies \eqref{eq:ADKEI}.

Similarly, note that by Bernstein's inequality (see Lemma \ref{lem:Bern})
\begin{align*}
\norm{P_{\leq N_0} u}_{L^4_\lambda}^4 & \lesssim \norm{P_{\leq 1}u}_{L^4_\lambda}^4 + \norm{P_{1 < \cdot < N_0}u}_{L^4_\lambda}^4 \\
& \lesssim \norm{P_{\leq 1}u}_{L^4_\lambda}^4 + \frac{N_0^{d+J}}{\lambda_k^{d}}\norm{\brak{\grad}^{-J}P_{> 1}u}_{L^2_\lambda}^4, 
\end{align*}
which shows that \eqref{ineq:WADII} implies \eqref{ineq:PEKillerWkI}.
\end{proof}

\section{Discussion and conjectures} \label{sec:Phys}

%\textcolor{red}{
We end with some discussion regarding the results and their relationship with existing physics literature. 
Let us first discuss the predicted Kolmogorov-Zakharov power spectrum as well as our choice of splitting $\sigma$ between the forcing and nonlinearity.  
If one simply uses the direct scaling, 
\begin{align}
-i \partial_t \psi = \Delta \psi - \abs{\psi}^2 \psi - i \nu D(\grad) \psi + i \sigma \sum_j g^\lambda_j dW_t^{(j)},  \label{def:NLSwk} 
\end{align}
the Kolmogorov-Zakharov (KZ) power spectrum (modulo perhaps logarithmic corrections for $1 \ll \abs{k}$), predicts that
\begin{align*}
\EE \abs{k}^{d-1} \abs{\hat{\psi}(k)}^2
\approx
\begin{cases}
\eps_{wa}^{1/2} \sigma^{2/3} k^{-1/3} & \quad N_F \ll \abs{k} \ll 1 \\
\eps^{1/3}_{ke} \sigma^{2/3} k^{-1} & \quad 1 \ll \abs{k} \ll N_D. 
\end{cases}
\end{align*}
See for example \cite{nazarenko2011wave,zakharov2012kolmogorov}, where this prediction is derived from the Wave-Kinetic Equation (WKE).  
Then, the driving conditions $\sigma \to 0$ and $\nu \sigma^{-2} \to 0$ would imply that $\lim_{\nu ,\sigma \to 0} \EE \norm{\psi}_{L^2}^2 = 0$.
Hence, if we want to potentially get a non-vanishing and non-divergent limit, we should study the unknown $u = \frac{\psi}{\sigma^{1/3}}$, which is precisely the choice we made to derive \eqref{def:NLS}. 
With this scaling, we expect the stationary measures $\set{\mu_{\lambda,\nu,\sigma}}$  in the $(\lambda,\nu,\sigma) \to (\infty,0,0)$ limit to converge to an invariant measure of the linear Schr\"odinger equation, even more specifically, one where all of the Fourier modes are jointly independent complex Gaussians. The nonlinear dynamics (and non-Gaussian statistics) would appear at the next order in $\sigma$. See \cite{nazarenko2011wave} for more details on this prediction.

%\textcolor{red}{
The weak nonlinearity conditions \emph{might} imply that $\nu$  cannot be arbitrarily small relative to $\sigma$.
For example, the KZ  spectrum predicts that $\EE \norm{P_{\geq 1} u}_{L^2_\lambda}^{2} \gtrsim \eps_{ke}^{1/3} \abs{\log \nu}$.
Hence, if one does not have $\sigma \ll \abs{\log \nu}^{-1}$, we could expect the Hamiltonian energy input per unit time to diverge as $\nu \to 0$ (see \eqref{eq:Hbal}). 
Similarly, note that in Theorems \ref{thm:Inverse} and \ref{thm:Direct}, there are no conditions which explicitly relate $\lambda$ to $\nu$ or $\sigma$ (as one may expect), however, there could be restrictions implicitly implied.

%\textcolor{red}{
It is important to emphasize that the validity of the KZ power spectrum does not follow (neither heuristically nor rigorously) from \eqref{eq:PiWA}, \eqref{eq:PiKE}, or \eqref{eq:PiH}. However, Theorems \ref{thm:Inverse} and \ref{thm:Direct} show that anything close to a KZ power spectrum would imply the flux balance laws \eqref{eq:PiWA}, \eqref{eq:PiKE}, or \eqref{eq:PiH}. 
Indeed, it is entirely possible that laws such as \eqref{eq:PiWA}, \eqref{eq:PiKE}, or \eqref{eq:PiH} will hold even when other aspects of wave turbulence, such as the WKE or KZ power spectra, fail in the statistically stationary regime (see \cite{ZP04} for numerical evidence of this possibility in some 1d wave turbulence problems). 
A future proof of the flux laws is likely to proceed by proving some of the basic principles of wave turbulence theory with much deeper methods which then in turn imply that the hypotheses of Theorems \ref{thm:Inverse} and \ref{thm:Direct} hold. 
This is how Yaglom's law \cite{Yaglom49} (the analogue flux balance law) was proved for Batchelor regime passive scalar turbulence in \cite{bedrossian2018lagrangian}.
The Batchelor spectrum \cite{Batchelor59} (the analogue of the KZ spectra) came later after significantly more effort: see \cite{bedrossian2022batchelor} which uses \cite{bedrossian2018lagrangian,bedrossian2022almost,bedrossian2021almost}. 
That the power spectrum was much harder to prove than the flux balance law seems likely to be a general phenomenon.

%\textcolor{red}{
Regarding the dissipation operator $D$, we note that the statistics within the inertial range are not expected to depend on  $D$ \cite{nazarenko2011wave,zakharov2012kolmogorov}.  %(a mathematical proof of this is wide open of course).  
It should be possible to generalize condition (ii) in Theorem \ref{thm:Direct} above to cover a variety of dissipation operators.
For example, more general operators of the form: 
\begin{align*}
D(k) = \frac{1}{\abs{k}^{2\alpha}} + \abs{k}^{2\gamma},  %\label{eq:GenD}
\end{align*}
where $\alpha \geq 0$ and $\gamma \geq 1$.
Common choices employed in numerical computations are of the hyper-viscous form, i.e. $\gamma > 1$ (see e.g. \cite{ZP04,LBNR12} and the references therein).

%\textcolor{red}{
To clarify the relationship between our mathematical setting and the physics literature, we have included an almost-precise mathematical conjecture regarding some of the important behaviors of \eqref{def:NLS} that are suggested by wave turbulence theory applied in the statistically stationary regime.
We have enumerated the statements in what we believe to be roughly increasing difficulty (and in decreasing likelihood to be correct exactly as stated).
%One important conjecture which is missing here is one that concerns positive Lyapunov exponents (and associated quantitative estimates) or estimates on the attractor dimension. 
\begin{conjecture}
%\textcolor{red}{
Let $\mu_{\lambda,\sigma,\nu}$ be a family of invariant measures of the NLS \eqref{def:NLS} for a sequence of parameters $(\lambda_k,\sigma_k,\nu_k) \to (\infty,0,0)$. Suppose that the external forcing satisfies Assumption \ref{Ass:g} and that the forcing is sufficiently non-degenerate such that the invariant measures are all unique. Finally, suppose that the strong driving condition holds:
\begin{align*}
\lim_{k \to \infty}\frac{\nu_k}{\sigma^2_k} = 0,
\end{align*}
and possibly other further restrictions on $(\lambda_k,\sigma_k,\nu_k)$, to be determined, that ensure a weakly nonlinear regime (e.g. \eqref{eq:L2van} and \eqref{eq:PEDvan}). 
Then the following holds:
%\textcolor{red}{
\begin{enumerate}
\item Full dissipation anomaly, in the sense that conditions \eqref{eq:ADWA}, \eqref{eq:L2van}, and \eqref{eq:PEkiller} all hold. 
\item There exists $s_\ast$, $s^\ast$, $r_\ast$, and $r^\ast$ such that: for $s > s_\ast$ and $s' < s^\ast$ we have 
\begin{align}
& \sup_{k \geq 0} \EE \norm{\abs{\grad}^{s} P_{\leq 1} u}_{L^2_\lambda}^2 < \infty \label{eq:lowbdconj} \\
& \sup_{k \geq 0} \EE \norm{\abs{\grad}^{s'} P_{\geq 1} u}_{L^2_\lambda}^2 < \infty, \label{eq:hibdconj}
\end{align}
while for all $r < r_\ast$ and all $r' > r^\ast$, 
\begin{align*}
& \sup_{k \geq 0} \EE \norm{\abs{\grad}^{r_\ast} P_{\leq 1} u}_{L^2_\lambda}^2 = \infty \\
& \sup_{k \geq 0} \EE \norm{\abs{\grad}^{r^\ast} P_{\geq 1} u}_{L^2_\lambda}^2 = \infty. 
\end{align*}
\item There is a measure $\mu^\infty \in \mathcal{P}(\mathbf{H})$ such that all of the Fourier modes are jointly independent Complex Gaussians which satisfies $\mu_{\lambda_k,\sigma_k,\nu_k} \rightharpoonup^\ast \mu^\infty$.
Here $\mathbf{H}$ is an appropriate Hilbert space of distributions defined on $\R^d$ which locally obey estimates of the type in \eqref{eq:lowbdconj} and \eqref{eq:hibdconj}. 
\item Claim (2) holds with $r_\ast = s_\ast = -1/3$, $r^\ast = s^\ast = 0$ (as predicted by the WKE through the KZ spectra).
\item There $\exists C_\ast,C^\ast,N_\ast,N^\ast$ independent of $\nu,\sigma$ and frequencies $\lim_{\nu,\sigma\to 0}N_D = \infty$ and $\lim_{\nu,\sigma \to 0} N_F = 0$ such that (possibly up to logarithmic corrections), 
\begin{align*}
\frac{1}{C_\ast} 2^{-2j/3} & \leq \EE \sum_{2^j \leq \abs{k} \leq 2^{j+1}} \abs{\hat{u}(k)}^2 \leq C_\ast 2^{-2j/3} &  N_F \leq 2^j \leq N_\ast \\
\frac{1}{C^\ast}  & \leq \EE \sum_{2^j \leq \abs{k} \leq 2^{j+1}} \abs{\hat{u}(k)}^2 \leq C^\ast &  N^\ast \leq 2^j \leq N_D.
\end{align*}
\end{enumerate}
\end{conjecture}

\appendix
\section{Fourier analysis conventions and Littlewood-Paley}
We define the Fourier transform on $\mathbb T^d_\lambda$, where below we denote $\mathbb Z_\lambda^d = (\frac{\mathbb Z}{\lambda})^d$ 
\begin{align*}
\hat{u}(k) & = \int_{\mathbb T_\lambda^d} u(x) e^{-i x \cdot k} \dx\\ 
\check{u}(x) & = \frac{1}{(2\pi\lambda)^d} \sum_{k \in \mathbb Z_\lambda^d} \hat{u}(k) e^{i k \cdot x}.
\end{align*}
With this convention we have for $u,g \in L^2(\mathbb T^d_\lambda)$,  
\begin{align*}
\frac{1}{(2\pi\lambda)^d }\sum_{k \in \mathbb Z^d_\lambda} \abs{\hat{u}(k)}^2 & = \int_{\mathbb T^d_\lambda} \abs{u(x)}^2 \dx \\
\frac{1}{(2\pi\lambda)^d }\sum_{k \in \mathbb Z^d_\lambda} \overline{\hat{g}(k)}\hat{u}(k) & = \int_{\mathbb T^d_\lambda} \overline{g(x)}u(x)  \dx \\ 
\widehat{f \ast g}(k) & = \hat{f}(k) \hat{g}(k) \\ 
\widehat{fg}(x) & = \frac{1}{(2\pi \lambda)^d} \hat{f} \ast \hat{g}. 
\end{align*}
We use the following Fourier multiplier notation: given any locally integrable function we can (at least formally) define the Fourier multiplier by 
\begin{align*}
\widehat{m(\grad) f}(k) = m(ik) \widehat{f}(k). 
\end{align*}

Next let us state our conventions surrounding Littlewood-Paley decomposition.s
Specifically, we set $\psi \in C^\infty(B(0,2))$ with $\psi(x) \equiv 1$ for $\abs{x} \leq 1$ and define for $k \in \mathbb Z_\lambda^d$, 
\begin{align*}
\widehat{P_{\leq N} u}(k)  := \widehat{u_{\leq N}}(k) := \psi\left(\frac{k}{N}\right) \hat{u}(k),      
\end{align*}
and the frequency `projections'
\begin{align*}
P_{N} u := u_{N} := u_{\leq 2N} - u_{\leq N} \\
P_{A \leq \cdot \leq B} := u_{A \leq \cdot \leq B} := u_{\leq B} - u_{\leq A}. 
\end{align*}
These are not true projections, however note that
\begin{align*}
P_{N/8 \leq \cdot \leq 8N} P_N u = P_N u. 
\end{align*}
Notice that for all functions with $\hat{u}(0) = 0$, we have the following 
\begin{align*}
\norm{u}_{L^2}^2 = \sum_{N \in 2^{\mathbb Z^d}} \norm{u_N}_{L^2}^2
\end{align*}
and if $\abs{j-k} \geq 2$ and $N = 2^{i}$, $M = 2^{j}$, then 
\begin{align*}
\brak{u_N,u_{M}} = 0. 
\end{align*}
Notice that we use a Littlewood-Paley decomposition which is uniform in $\lambda$, that is, the dyadic shells are chosen independent of $\lambda$, so that as $\lambda$ varies, the number of frequencies included in a given Littlewood-Paley `projection' varies. 
We have the following uniform-in-$\lambda$ Bernstein's inequalities.

\begin{lemma}[Bernstein's inequalities] \label{lem:Bern}
The following inequalities hold for all $N \in 2^{\mathbb Z}$ uniformly in $\lambda$ and $1 \leq p \leq q \leq \infty$, 
\begin{align}
\norm{u_{\leq N}}_{L^p} & \lesssim \norm{u}_{L^p} \label{ineq:B1} \\
\norm{\grad u_{\leq N}}_{L^p} & \lesssim N\norm{u}_{L^p} \label{ineq:B2} \\
\norm{u_{N}}_{L^p} & \lesssim \norm{u}_{L^p} \label{ineq:B3} \\
\norm{N^s u_{N}}_{L^p} & \approx \norm{\abs{\grad}^s u_N}_{L^p} \label{ineq:B4} \\
\norm{u_N}_{L^q} & \lesssim \left(\frac{N}{\lambda}\right)^{d\left(\frac{1}{p} - \frac{1}{q}\right)} \norm{u_N}_{L^p}. \label{ineq:B5}
\end{align}
\end{lemma}
\begin{proof}
First note that
\begin{align*}
u_{\leq N} = \frac{1}{N^d} \check{\psi}\left( \frac{\cdot }{ N} \right) \ast u. 
\end{align*}
Note that
\begin{align*}
\check{\psi}\left( x \right) = \frac{1}{(2\pi \lambda)^d} \sum_{k \in \mathbb Z_\lambda^d} \psi(k) e^{i k \cdot x}.
\end{align*} 
We will show that $N^{-d}\check{\psi}(N^{-1}\cdot)$ makes an approximation of the identity uniformly in $\lambda$.
First, note that
\begin{align*}
\grad_x^\beta \check{\psi}\left( x \right) = \frac{1}{(2\pi \lambda)^d} \sum_{k \in \mathbb Z_\lambda^d} (ik)^{\otimes \beta} \psi(k) e^{i k \cdot x},   
\end{align*}
and hence (note no constants depend on $\lambda$)
\begin{align*}
\abs{\grad_x^\beta \check{\psi}(x)} \lesssim_d \frac{2^\beta}{\lambda^d} \sum_{k \in \mathbb Z_\lambda^d : \abs{k} \leq 2} 1 \approx_d  2^\beta.  
\end{align*}
By summation by parts, (and $\abs{x} \leq \pi$) we have
\begin{align*}
ix_j \check{\psi}\left( x \right) & = \frac{1}{(2\pi \lambda)^d} \sum_{k \in \mathbb Z_\lambda^d} \psi(k) ix_j e^{i k \cdot x} \\
& = \frac{1}{(2\pi \lambda)^d} \sum_{k \in \mathbb Z_\lambda^d} \psi(k) \frac{ix_j}{\left(e^{i x_j/(\lambda)} - 1 \right)} \left(e^{ i(k+\lambda^{-1} e_j)\cdot x} - e^{ik \cdot x}\right) \\
& = - \frac{1}{(2\pi \lambda)^d} \sum_{k \in \mathbb Z_\lambda^d} \left(\psi(k) - \psi(k - \lambda^{-1}e_j) \right) \frac{ix_j}{\left(e^{i x_j/(\lambda)} - 1 \right)} e^{ik \cdot x}.
\end{align*}
Then, notice that 
\begin{align*}
& \abs{\psi(k) - \psi(k - \lambda^{-1}e_j)} \lesssim \frac{1}{\lambda} \\
& \abs{\frac{ix_j}{\left(e^{i x_j/(\lambda)} - 1 \right)}} \lesssim \lambda, 
\end{align*}
and so, similar to above, we have uniformly in $\lambda$, 
\begin{align*}
\abs{x_j \check{\psi}(x)} \lesssim_d 1.
\end{align*} 
Iterating this for higher powers of $x^\beta$ simply gives higher and higher order finite-differences of $\psi$, resulting in a similar $O(\lambda^{-\beta})$ to balance each additional power.
One can similarly obtain localization estimates on the derivatives of $\check{\psi}$. 
It follows that the sequence $N^{-d} \psi(\cdot N^{-1})$ (along with the average zero property) is a smooth approximation to the identity that satisfies all the properties one desires of a mollifier. 
This immediately implies \eqref{ineq:B1}, \eqref{ineq:B2}, and \eqref{ineq:B3} by Young's convolution inequality.  

Having seen how $\mathbb T^d_\lambda$ is dealt with, the proofs of \eqref{ineq:B4} and \eqref{ineq:B5} proceed in the usual manner. 
\end{proof}

\bibliographystyle{abbrv}
\bibliography{bibliography}
\end{document}

%% file: macros.tex
\newcommand{\eps}{\epsilon}

\newcommand{\grad}{\nabla}

\newcommand{\norm}[1]{\left|\left| #1 \right|\right|}
\newcommand{\abs}[1]{\left| #1 \right|}
\newcommand{\set}[1]{\left\{ #1 \right\}}
\newcommand{\brak}[1]{\left\langle #1 \right\rangle} 

\newcommand{\R}{\mathbb{R}}

\newcommand{\dee}{\mathrm{d}}

\newcommand{\dx}{\dee x}

\usepackage{bbm}

\newcommand{\EE}{\mathbf E}
\newcommand{\PP}{\mathbf P}

\newtheorem{theorem}{Theorem}[section]
\newtheorem{proposition}[theorem]{Proposition}

\newtheorem{lemma}[theorem]{Lemma}
\newtheorem*{lemma*}{Lemma}

\newtheorem{assumption}{Assumption}
\newtheorem{conjecture}{Conjecture}

\theoremstyle{definition}
\newtheorem{definition}[theorem]{Definition}
\newtheorem{remark}[theorem]{Remark}
\newtheorem{example}[theorem]{Example}